\documentclass[10pt, a4paper, twoside]{amsart}
\usepackage{amsthm}
\usepackage{amsmath}
\usepackage{amssymb}
\usepackage{mathtools}
\usepackage[all]{xy}
\usepackage{indentfirst}
\usepackage{comment}

\newtheorem{thm}{\bf Theorem}[section]
\newtheorem{prop}[thm]{\bf Proposition}
\newtheorem{lem}[thm]{\bf Lemma}

\newtheorem{cor}[thm]{\bf Corollary}

\newtheorem*{thm*}{\bf Theorem}
\newtheorem*{cor*}{\bf Corollary}

\theoremstyle{definition}
\newtheorem{df}[thm]{\bf Definition}
\newtheorem{rem}[thm]{\it Remark}

\newtheorem*{df*}{\bf Definition}
\newtheorem*{not*}{\bf Notation and Convention}
\newtheorem*{ack*}{\bf Acknowledgements}
\newtheorem*{dfrem*}{\bf Definition and Remark}

\newtheorem*{nota*}{\bf Notation}

\newtheorem*{remone*}{\it Remark}
\newtheorem*{remtwo*}{\it Remark}
\newtheorem*{remthree*}{\it Remark}
\newtheorem*{remfour*}{\it Remark}

\def\P{\mathbb{P}}
\def\C{\mathbb{C}}

\def\Q{\mathbb{Q}}

\def\R{\mathbb{R}}

\DeclareMathOperator{\can}{can}

\DeclareMathOperator{\Aut}{Aut}
\DeclareMathOperator{\Bir}{Bir}

\DeclareMathOperator{\Proj}{Proj}
\DeclareMathOperator{\Sing}{Sing}
\DeclareMathOperator{\val}{val}

\DeclareMathOperator{\Jac}{Jac}

\DeclareMathOperator{\Sec}{Sec}
\DeclareMathOperator{\Bs}{Bs}

\DeclareMathOperator{\Sym}{Sym}

\makeindex

\subjclass[2010]{Primary: 14E08, Secondary: 14B05, 14E05, 14J45}
\keywords{Birational rigidity}
\title[\tiny Birational rigidity of complete intersections]
{Birational rigidity of complete intersections}
\author{Fumiaki Suzuki}
\address{Graduate School of Mathematical Sciences, University of Tokyo,
Meguro-ku, Tokyo, 153-9814, Japan.}
\email{fsuzuki@ms.u-tokyo.ac.jp}

\begin{document}
\maketitle

\begin{abstract}
We prove that every smooth complete intersection $X=X_{d_{1}, \cdots, d_{s}}\subset \P^{\sum_{i=1}^{s}d_{i}}$
defined by $s$ hypersurfaces of degree $d_{1}, \cdots, d_{s}$ is birationally superrigid 
if $5s +1\leq \frac{2(\sum_{i=1}^{s}d_{i}+1)}{\sqrt{\prod_{i=1}^{s}d_{i}}}$.
In particular, $X$ is non-rational and $\Bir(X)=\Aut(X)$.
We also prove birational superrigidity of singular complete intersections with similar numerical condition.
These extend the results proved by Tommaso de Fernex.
\end{abstract}

\section*{Introduction}
Throughout this paper, we work over the field of complex numbers $\C$.
A complete intersection of type $X_{d_{1},\cdots, d_{s}}\subset \P^{N}$, which is defined by $s$ hypersurfaces of degree $d_{1}, \cdots, d_{s}$ in a projective space $\P^{N}$,
is $\Q$-Fano, i.e. normal, $\Q$-factorial, terminal and having an ample anti-canonical divisor,
if $\sum_{i=1}^{s}d_{i}\leq N$ and it has only mild singularities.
Then it is rationally connected by the results of K\'ollar-Miyaoka-Mori \cite{KMM}, Zhang \cite{Z} and Hacon-Mckernan \cite{HM}.
A natural problem is to determine its rationality.
If its dimension is at most $2$ or if its degree is so, then it is rational.
How about the remaining cases?
In these cases, its Picard number is $1$ by the Lefschetz hyperplane section theorem.
We mean by a $\Q$-Fano variety that of Picard number $1$ in what follows.

Given a $\Q$-Fano variety, one of the most effective ways to prove its {\it non}-rationality is proving its birational superrigidity.
We recall that a $\Q$-Fano variety $X$ is called {\it birationally superrigid} if any birational map to the source of another Mori fiber space is isomorphism.
It implies that $X$ is non-rational and $\Bir(X) = \Aut(X)$.

Return to the initial problem and ask the following question:  in the remaining cases, which $\Q$-Fano complete intersections are birationally superrigid?
By general linear projections, those of index $\geq 2$ always have non-trivial birational Mori fiber space structures.
So we only consider the index $1$ case; let $N=\sum_{i=1}^{s}d_{i}$ in what follows.

First, let $s=1$.
Iskovskih and Manin proved that every smooth quartic $3$-fold $X_{4}\subset \P^{4}$ is non-rational by proving that any birational ones are isomorphic to each other in the paper \cite{IM},
where the notion of birational superrigidity has its origin.
This gave the negative answer to L\"uroth problem
together with the paper \cite{CG} by Clemens and Griffiths.
Then, after the works of Iskovskih-Manin, Pukhlikov, Chel'tsov and de Fernex-Ein-Musta\c t\v a\ \cite{Ch1, dFEM2, IM, Puk, Puk6},
de Fernex proved in \cite{dF1} (see also \cite{dF2} for an {\it erratum} with an amended proof to accompany \cite{dF1}) that every smooth hypersurface $X_{N}\subset \P^{N}$ is birationally superrigid for $N\geq 4$.
This completes the list of birationally superrigid smooth hypersurfaces.
He also proved birational superrigidity of a large class of singular Fano hypersurfaces of index $1$ in \cite{dF3}
(see \cite{Ch4, CM, Me, Puk1, Puk8, Puk7, Sh} for related results on singular hypersurfaces).

In this paper, we extend the results of de Fernex in \cite{dF1, dF3} for $s\geq 2$.
Before stating a main theorem, we briefly summerize known results.
For $s\geq 2$, birational superrigidity and birational rigidity (see \cite[Definition 1.3]{C} for the definition; a slightly weaker notion than birational superrigidity, sufficient for non-rationality though) are known only when a complete intersection is smooth and satisfies general conditions.
The following is the list, where the first and second ones are about birational superrigidity and the third one is about birational rigidity:
\begin{itemize}
\item smooth complete intersections $X_{d_{1},\cdots, d_{s}}\subset \P^{\sum_{i=1}^{s}d_{i}}$ of dimension $\geq 12$ which satisfy so-called regularity conditions (see \cite{Puk9} for the definition),
except  three infinite series $X_{2,\cdots,2}, X_{2,\cdots, 2,3}$ and $X_{2,\cdots, 2,4}$, by Pukhlikov (see \cite{Puk5, Puk3, Puk4}),
\item smooth complete intersections $X_{2,4}\subset \P^{6}$ not containing planes by Chel'tsov (see \cite{Ch2}),
\item general smooth complete intersections $X_{2,3}\subset \P^{5}$ by Iskovskih-Pukhlikov (see \cite{IP}, and see \cite[Chapter 3, Remark 1]{IP} for what we exactly mean by ``general'' here).
\end{itemize}
Note that no explicit examples which satisfy these conditions have been obtained so far.
In addition, in the following cases, non-rationality is proved by Beauville, using intermediate jacobians:
\begin{itemize}
\item every smooth complete intersection $X_{2,2,2}\subset \P^{8}$ (see \cite{B1}),
\item general smooth complete intersections $X_{2,3}\subset \P^{5}$ (see \cite{B1}),
\item the smooth complete intersection defined by $\sum_{i=0}^{6}X_{i}=\sum_{i=0}^{6}X_{i}^{2}=\sum_{i=0}^{6}X_{i}^{3}$
$=0$ in $\P^{6}$ (see \cite{B2}).
\end{itemize}
No rational members are known.

To state a main theorem, we recall the following definition of singularities, which is a modification of that introduced in \cite{dF3}.

\begin{df}
Let $p\in X$ be a germ of a variety.
For any triple of integers $(\delta, \nu, k)$ with $\delta \geq -1, \nu \geq 1$ and $k\geq 0$,
we say that $p$ is a {\it singularity of type $(\delta, \nu, k)$}
if the singular locus of $X$ has dimension at most $\delta$,
and given a general complete intersection $Y\subset X$ of codimension $\min \left\{\delta + k, \dim X \right\}$ through $p$,
the $(\nu -1)$-th power of the maximal ideal $\mathfrak{m}_{Y, p}\subset \mathcal{O}_{Y,p}$ is contained in the integral closure of the Jacobian ideal $\Jac_{Y}$ of $Y$.
We use the convention that $p$ is a singularity of type $(-1, 1, k)$ for any $k\geq 0$ if $p$ is a smooth point.
\end{df}

For $s$ positive integers $d_{1},\cdots,d_{s}$, set $c_{s}(d_{1},\cdots,d_{s})=\frac{2(\sum_{i=1}^{s}d_{i}+1)}{\sqrt{\prod_{i=1}^{s}d_{i}}} - 5s$ in what follows.
Our main theorem is the following.

\begin{thm}\label{mainthm}
Let $d_{1}, \cdots, d_{s}\geq 2, \delta \geq -1$ and $\nu \geq 1$ be integers which satisfy
\[2\delta + \nu +2 \leq c_{s}(d_{1},\cdots,d_{s}).\]
Then every complete intersection $X=X_{d_{1},\cdots, d_{s}}\subset \P^{\sum_{i=1}^{s}d_{i}}$
with only singularities of type $(\delta, \nu, 2s)$
is $\Q$-Fano and birationally superrigid. 
In particular, $X$ is non-rational and $\Bir(X) = \Aut(X)$.
\end{thm}

We give a few corollaries, to see which complete intersetions are covered by Theorem \ref{mainthm}. 
First we consider the smooth case.
Then, if we fix $d_{1},\cdots, d_{s-1}$, the inequality in Theorem \ref{mainthm} is satisfied for sufficiently large $d_{s}$.
The following are the simplest examples.

\begin{cor}
Every smooth complete intersection
\[X=X_{2, d}\subset \P^{d+2}, X_{3,d}\subset \P^{d+3}, X_{4,d}\subset \P^{d+4}, X_{2,2,d}\subset \P^{d+4}\]
is birationally superrigid for $d\geq 55, 83, 111, 246$ respectively.
\end{cor}

In next three corollaries, we consider the isolated hypersurface singularity case, i.e. $\delta = 0$ and $\dim \mathfrak{m}_{p}/\mathfrak{m}_{p}^{2} = \dim X +1$ for every $p \in \Sing(X)$.

Recall that an isolated hypersurface singularity is called {\it semi-homogeneous} if its tangent cone is smooth away from the vertex.
If we use \cite[Proposition 2.4]{dF3}, Theorem \ref{mainthm} implies the following.

\begin{cor}\label{cor2}
Let $d_{1}, \cdots, d_{s}\geq 2$ be positive integers and $X= X_{d_{1},\cdots, d_{s}}\subset \P^{\sum_{i=1}^{s}d_{i}}$ be a singular complete intersection
with isolated semi-homogeneous hypersurface singularities.
If
\[
e_{p}(X)\leq c_{s}(d_{1},\cdots, d_{s}) -2
\]
for every $p\in \Sing(X)$,
then $X$ is $\Q$-Fano and birationally superrigid.
\end{cor}

Recall that for an isolated hypersurface singularity $p\in X$, its Tyurina number is defined by $\tau_{p}(X)= \dim \mathcal{O}_{X,p}/\Jac_{X}$.
For $m\geq 1$, set $\tau_{p}^{(m)}(X)$ to be the Tyurina number of a general complete intersection of $X$ of codimension $m$ through $p$.
By the same argument as in the proof of \cite[Corollary 1.5]{dF3}, Theorem \ref{mainthm} implies the following.

\begin{cor}\label{cor3}
Let $d_{1}, \cdots, d_{s}\geq2$ be positive integers and $X=X_{d_{1},\cdots,d_{s}}\subset \P^{\sum_{i=1}^{s}d_{i}}$ be a singular complete intersection with isolated hypersurface singularities.
If
\[
\min \left\{ \tau_{p}(X),\tau_{p}^{(1)}(X)\cdots, \tau^{(2s)}_{p}(X)\right\}\leq c_{s}(d_{1},\cdots, d_{s})-3
\]
for every $p\in \Sing(X)$,
then $X$ is $\Q$-Fano and birationally superrigid.
\end{cor}

For $s$ positive integers $d_{1},\cdots,d_{s}$, set $c'_{s}(d_{1},\cdots, d_{s})= \left(\prod_{i=1}^{s}d_{i}\right)\cdot (\sum_{i_{1}+\cdots+i_{s}=\dim X}$\\$(d_{1}-1)^{i_{1}}\cdots (d_{s}-1)^{i_{s}})$.
Let $X=X_{d_{1},\cdots, d_{s}}\subset \P^{N}$.
We denote by $X^{\vee} \subset (\P^{N})^{\vee}$ the dual variety of $X$.
It is known that $X^{\vee}$ is a hypersurface of degree $c'_{s}(d_{1},\cdots,d_{s})$ if $X$ is smooth and non-linear.
By the same argument as in the proof of \cite[Corollary 1.6]{dF3}, using the generalized Teissier-Pl\"ucker formula in \cite[Theorem 1]{Kle}, Theorem \ref{mainthm} implies the following.

\begin{cor}\label{cor4}
Let $d_{1}, \cdots, d_{s}\geq2$ be positive integers and $X=X_{d_{1},\cdots,d_{s}}\subset \P^{\sum_{i=1}^{s}d_{i}}$ be a singular complete intersection with $t$ isolated hypersurface singularities.
If $X^{\vee}$ is a hypersurface of degree
\[
\deg X^{\vee} \geq c'_{s}(d_{1},\cdots,d_{s}) -( 2 c_{s}(d_{1},\cdots,d_{s}) + 2t -8),
\]
then $X$ is $\Q$-Fano and birationally superrigid.
\end{cor}

Section $1$ is devoted to review definitions and basic facts about Segre classes, Chern classes and Samuel multiplicities.
In Section $2$, we prove Proposition \ref{prop} as a key to prove Theorem \ref{mainthm},
which is a generalization of Pukhlikov's multiplicity bounds in \cite[Proposition 5]{Puk2} to complete intersections.
In Section $3$, we prove Theorem \ref{mainthm}.

\begin{not*}
A variety is assumed to be irreducible and reduced.
We say that a property $P$ holds for a general point in a variety if there exists a non-empty open subset in the variety such that $P$ holds for every point of the open set.
For a pure-dimensional scheme $X$ of finite type, denote by $[X]$ its fundamental cycle, and by $e_{p}(X)$ its Samuel multiplicity at a closed point $p$ in $X$ (see Definition \ref{dfsm}).
In Proposition \ref{prop}, we use the following notations:
\begin{itemize}
\item for pure-dimensional cycles $\alpha_{1}, \alpha_{2}$ on a scheme $X$, we write $\alpha_{1} \sim \alpha_{2}$ if $\alpha_{1}$ and $\alpha_{2}$ are rationally equivalent;
\item for pure-dimensional cycles $\beta_{1}, \beta_{2}$ intersecting properly on $X$ and an irreducible component $T$ of the intersection,
denote by $i(T, \beta_{1} \cdot \beta_{2} ; X)$ the intersection multiplicity of $T$ in $\beta_{1} \cdot \beta_{2}$ whenever the intersection product $\beta_{1} \cdot \beta_{2}$ is defined;
\item for a cycle $\gamma$, denote by $|\gamma|$ the support of $\gamma$, which is the union of the subvarieties appearing with non-zero coefficient in $\gamma$,
\item for a closed subscheme $Z$ of $X$, denote by $s(Z, X)$ the Segre class of $Z$ in $X$ (see Definition \ref{dfsc}),
\item for a vector bundle $E$ on $X$, denote by $c(E)$ the total Chern class of $E$ (see Definition \ref{dfsc}) and by $c(E)\cap \zeta$ its cup-product with a cycle $\zeta$,
\item for a projective variety $Y$ embedded in some projective space $\P^{N}$, denote by $c_{1}(\mathcal{O}_{Y}(1))$ the first Chern class of  a hyperplane section,
\item for projective varieties $U, V$ and a point $q$ in a projective space $\P^{N}$,
\[
J(U, V)  = \overline{\bigcup_{u\in U, v\in V} \langle u,v \rangle},\
\Sec(U) = J(U, U),\ 
C(q, U) = J(q, U),
\]
 and for a closed subset $W$ (resp. $W'$) with irreducible components $W_{1},\cdots, W_{s}$ (resp. $W_{1}', \cdots, W_{t}'$), as sets,
 \[
 J(W, W') = \bigcup_{i=1}^{s}\bigcup_{j=1}^{t} J(W_{i}, W_{j}'),\ \ C(q, W) = \bigcup_{i=1}^{s}C(q, W_{i}).
 \]
 \end{itemize}
For the definitions and basic properties of fundamental cycles, rational equivalence, intersection products and intersection multiplicities, we follow \cite{Ful}.
\end{not*}

\begin{ack*}
The author wishes to express his gratitude to his supervisor Professor Yujiro Kawamata for his encouragement and valuable advice.
The author is grateful to Professor Tommaso de Fernex for sending his drafts of the papers \cite{dF2, dF3} and for helpful suggestions. 
The author wishes to thank Akihiro Kanemitsu, Chen Jiang, Pu Cao and Yusuke Nakamura for careful reading of the manuscript and helpful suggestions.
This paper is an extension of the master thesis of the author at University of Tokyo.
This work was supported by the Program for Leading Graduate 
Schools, MEXT, Japan.
\end{ack*}

\section{Preliminaries}

\subsection{Segre classes and Chern classes}

We recall the notion of Segre classes and Chern classes, following \cite[Chapter 3 and 4]{Ful}.

\begin{df}\label{dfsc}
Let $X$ be a scheme and $E$ be a vector bundle of rank $e+1$ on $X$.
Set $P(E)=\Proj(\Sym E^{\vee})$.
Define the {\it total Segre class $s(E)$ of $E$} as follows:
\[
s(E)\cap \alpha = \sum_{i\geq0} p_{*} (c_{1}(\mathcal{O}(1))^{e+i} \cap p^{*}\alpha ),
\]
where $p\colon P(E)\rightarrow X$ is the projection, $\mathcal{O}(1)$ is the tautological line bundle and $\alpha$ is an arbitrary cycle.
Define the {\it total Chern class} of $E$ by
\[
c(E) = s(E)^{-1}.
\]

Let $Z$ be a closed subscheme of $X$.
Set $P(C_{Z}X\oplus 1) = \Proj\left((\oplus_{i\geq 0}\mathcal{I}_{Z}^{i}/\mathcal{I}_{Z}^{i+1})[t]\right)$ 
where $t$ is a variable.
Define the {\it Segre class $s(Z, X)$ of $Z$ in $X$} as follows:
\[
s(Z, X) = \sum_{i\geq 1} \pi_{*}( c_{1}(\mathcal{O}(1))^{i}\cap P(C_{Z}X \oplus 1)),
\]
where $\pi \colon P(C_{Z}X\oplus 1) \rightarrow Z$ is the projection
and $\mathcal{O}(1)$ is the tautological line bundle.
\end{df}

If $Z$ is regularly imbedded in $X$, then
\[
s(Z, X) =  c(N_{Z}X)^{-1}\cap [Z]
\]
by \cite[Proposition 4.1]{Ful}.

\subsection{Samuel multiplicities}

\begin{df}\label{dfsm}
The {\it Samuel multiplicity} or {\it multiplicity} of a pure-dimensional scheme $X$ of finite type at a closed point $p\in X$
is defined to be the Samuel multiplicity of the maximal ideal $\mathfrak{m}_{X, p} \subset \mathcal{O}_{X, p}$
\[e_{p}(X) =  e(\mathfrak{m}_{X, p}) = \lim_{t\rightarrow \infty} \frac{n! l(\mathcal{O}_{X, p}/\mathfrak{m}_{X, p}^{t+1})}{t^{n}},\]
where $n=\dim \mathcal{O}_{X, p}$.
Here, $e_{p}(X)$ agrees with the coefficient of $[p]$ in the Segre class $s(p, X)$ (see \cite[Example 4.3.4]{Ful}).
For an irreducible subvariety $S$ of $X$, define
\[e_{S}(X) =  \min \left\{e_{p}(X) \mid p \in S \right\}.
\]
This is well-defined by the upper-semicontinuity of multiplicities \cite{Ben}.
We extend the definition of the multiplicity linearly to an arbitrary cycle
where we use the convention $e_{p}(X) = 0$ if $p\not\in X$.
\end{df}

By \cite[Lemma 4.2]{Ful},
\[
e_{p}(X) = e_{p}([X])
\]
for every $p\in X$.
Thus we identify the scheme $X$ and the cycle $[X]$ when we deal with its multiplicity (and also its degree).

Samuel multiplicities satisfies the following property when we cut down a given pure-dimensional Cohen-Macaulay closed scheme by a hyperplane.

\begin{prop}[\cite{dFEM2}, Proposition 4.5]\label{mult}
Let $X$ be a positive, pure-dimensional Cohen-Macaulay closed subscheme in $\P^{N}$.
Then for an arbitrary hyperplane $\mathcal{H} \subset (\mathbb{P}^N)^{\vee}$, if $H\in \mathcal{H}$ is general, 
\[e_{p}(X\cap H) = e_{p}(X)\]
 for every $p\in X\cap H$.
\end{prop}

\section{A generalization of Pukhlikov's multiplicity bounds}

In this section, we prove a following key proposition.

\begin{prop}\label{prop}
Let $X$ be a complete intersection in $\P^{N}$ defined by $s$ hypersurfaces
and $\alpha$ be an effective cycle on $X$ of pure codimension $k$
such that $\alpha \sim m\cdot c_{1}(\mathcal{O}_{X}(1))^{k} \cap [X]$.
Assume either that $X$ is smooth or $ks + \dim \Sing(X) + 1<N$.
Then $e_{S}(\alpha)\leq m$ for every closed subvariety $S\subset X$ of dimension $\geq ks$ not meeting the singular locus of $X$.
\end{prop}

\begin{rem}
Proposition \ref{prop} is proved when $s=1$ in \cite[Proposition 5]{Puk2} and when $k=1$ in \cite[Lemma 13]{Ch3}.
\end{rem}

\begin{proof}[Proof of Proposition \ref{prop}]
\renewcommand{\theequation}{\fnsymbol{equation}}
\setcounter{equation}{1}
Let $X_{1}, \cdots, X_{s}$ be hypersurfaces in $\P^{N}$ defining $X$ with $\deg X_{i} = d_{i}$.
We take cones in $\P^{N+1}$ so that residual schemes can be defined in Step $1$ below, which is essential when we construct residual intersection cycles.
Fix a closed point $p\in \P^{N+1}\setminus \P^{N}$.
Set
\[
X_{i}' = C(p, X_{i}) \subset \P^{N+1}
\]
for each $i=1,\cdots, s$, and 
\[
X' = C(p, X) \subset \P^{N+1}.
\]
Then $X_{1}', \cdots, X_{s}'$ are hypersurfaces in $\P^{N+1}$ with $\deg X_{i}' = d_{i}$
and $X'$ is a complete intersection defined by $X_{1}',\cdots, X_{s}'$.\\

{\it Step 1:}
We use the method of multiple residual intersection as in the proof of \cite[Proposition 3]{Puk2} or \cite[Theorem]{R},
and this step is devoted to its preparation.
First we explain the construction of residual intersection cycles, following \cite[Section 9.2]{Ful}.
Let $T$ be a closed subvariety of $X$.
Take a closed point $q\in \P^{N+1}\setminus \bigcup_{i=1}^{s}X_{i}'$ and let $C=C(q,T)$.
Then $C$ is a $(\dim T+1)$-dimensional variety and $T$ is a hyperplane section of $C$.
Hence $T$ is a Cartier divisor on $C$ and
so we can define the residual scheme $R(q, T)$ to $T$ in $C\cap X'$.
Let $R^{1}(q, T), \cdots, R^{s}(q, T)$ be the residual schemes to $T$ in $C\cap X_{1}', \cdots, C\cap X_{s}'$, then $R(q, T)= \bigcap_{i=1}^{s}R^{i}(q,T)$ as a closed subscheme of $C$.
Consider a diagram
\[\xymatrix{
&R(q,T)\ar[d]^{b} & \\
T\ar[r]^{a}&C\cap X' \ar[d]^{g} \ar[r]^{j} & C \ar[d]^{f}\\
&X' \ar[r]^{i} & \P^{N+1} \\
}
\]
where
\begin{itemize}
\item $i, j, a, b, f, g$ are natural closed imbeddings,
\item the square is a fiber square.
\end{itemize}
Note that $i$ is a regular imbedding of codimesion $s$ and $ja$ imbeds $T$ as a Cartier divisor on $C$.
Formally define the residual intersection class
\[
\R(q, T) = \left\{ c(N\otimes \mathcal{O}(-T)) \cap s(R(q,T), C)  \right\}_{\dim T -s+1}
\]
in $A_{\dim T-s+1}(R(q, T))$, where $N= g^{*}N_{X'}\P^{N+1}$ and $\mathcal{O}(-T) = j^{*}\mathcal{O}_{C}(-T)$.
By the residual intersection theorem (see \cite[Theorem 9.2]{Ful}),
\[
C\cdot X' = \left\{c(N)\cap s(T, C)  \right\}_{\dim T -s+1} + \R(q,T)
\]
in $A_{\dim T-s+1}(C\cap X')$.
Thus
\[
\R(q,T)\sim \left(\prod_{i=1}^{s}(d_{i}-1)\right) \cdot c_{1}(\mathcal{O}_{C}(1))^{s}\cap [C]
\]
in $A_{\dim T-s+1}(C)$.
(Note that we can also define $\R(q, T)$ naively as follows:
\[
\R(q, T) = R^{1}(q, T)\cdot  \ldots \cdot R^{s}(q,T).
\]
Both definitions coincide by \cite[Proposition 6.1 (a) and Example 6.5.1 (b)]{Ful}.
Then the above rational equivalence is established again immediately since 
\[
R^{i}(q, T) \sim (d_{i}-1)\cdot c_{1}(\mathcal{O}_{C}(1))\cap [C]
\]
for each $i=1,\cdots, s$.)
We will see $\R(q, T)$ as the cycle class of $X'$.
We extend the definition of the residual intersection to arbitrary pure-dimensional cycles on $X$ linearly.

Next we prove fundamental properties of residual intersections and polar loci of linear projections.
They are related to each other
and we can use polar loci to estimate the dimension of intersection of residual intersections with closed subvarieties.
Fix a homogeneous coordinate $[X_{0}:\cdots:X_{N+1}] \in \P^{N+1}$.
For a closed point $q=[q_{0}:\cdots:q_{N+1}] \in \P^{N+1}\setminus \left( \bigcup_{i=1}^{s} X_{i}' \right)$ and each $i=1, \cdots, s$,
define the polar locus of $X_{i}'$ by
\[P^{i}(q)= \left\{  \sum_{k=0}^{N+1}q_{k}\frac{\partial f_{i}}{\partial X_{k}} = 0 \right\} \subset X_{i}',
\]
where $f_{i}$ is the defining equation of $X_{i}'$,
and define the polar locus of $X'$ by
\[P(q) = \left\{  \sum_{k=0}^{N+1}q_{k}\frac{\partial f_{1}}{\partial X_{k}}=\cdots = \sum_{k=0}^{N+1}q_{k}\frac{\partial f_{s}}{\partial X_{k}} = 0  \right\} \subset X'.
\]
By definition,
\[
P(q) \cap (X')^{sm} = \left\{x \mid q \in H_{x} \right\} \subset (X')^{sm}
\]
as a set, where $H_{x} \subset \P^{N+1}$ is the embedded tangent space of $X'$ at $x$.

\begin{lem}\label{lem}
Let $T\subset X$ be a closed subvariety with $T\cap \Sing(X)=\emptyset$,
and $U, V\subset X'$ be closed subvarieties.
Then the following hold for a general point $q\in \P^{N+1}\setminus (\bigcup_{i=1}^{s}X_{i}')$
and a general $k$-tuple $(q_{1},\cdots,q_{k})\in \left(\P^{N+1}\setminus (\bigcup_{i=1}^{s}X_{i}')\right)^{k}\colon$
\begin{enumerate}
\item[$(1)$]$T\cap R(q, T) = T \cap P(q)$ as a set.
\item[$(2)$]
$\left\{
\begin{array}{ll}
\dim U\cap  T\cap R(q, T)= \dim U \cap T -s & \text{ if $\dim U\cap T \geq s$,}\\
U\cap T \cap R(q, T)= \emptyset & \text{ otherwise.}
\end{array}
\right.$
\item[$(3)$]$\dim U \cap R(q, V) \setminus V \leq \dim U + \dim V -N$, where we use the convention $\dim(\emptyset) = -\infty$.
\item[$(4)$]If $\dim T=ks$, then the number of the points of $T\cap \bigcap_{j=1}^{k} P(q_{j})$ is the same as 
$\left(\prod_{i=1}^{s}(d_{i}-1)^{k}\right) \deg T$.
\end{enumerate}
\end{lem}
\begin{proof}
$(1)$
It is enough to show for each $i=1,\cdots, s$,
\[T \cap R^{i}(q,T) = T \cap P^{i}(q)
\]
for a general $q\in \P^{N+1}\setminus (\bigcup_{i=1}^{s}X_{i}')$.
This follows from Pukhlikov's argument in the proof of \cite[Lemma 3]{Puk1}
since the secant variety $\Sec(T)$ of $T$ is contained in $\P^{N}$.

For dimension estimates in the following proofs, we freely use the generic flatness (for example, see \cite[Theorem 24.1]{M}).

$(2)$
Let $W\subset U\cap T$ be an irreducible component.
By the assumption, $T$ is contained in the smooth locus of $X'$, and so is $W$.
The incidence set
\[I_{W} = \left\{(x,q) \mid q\in H_{x} \right\} \subset W\times \P^{N+1}\]
is irreducible and $\dim I_{W} = \dim W + N-s +1$.
For each $q\in \P^{N+1}$, the fiber of the projection over $q$ is $W\cap P(q)$ by the definition of $I_{W}$,
and $W\cap P(q)$ is non-empty if $\dim W \geq s$ since $P(q)$ is defined by $s$ hypersurfaces.
Thus the second projection is surjective if and only if $\dim W \geq s$,
so
\[
\left\{
\begin{array}{ll}
\dim W\cap P(q) = \dim W - s & \text{ if $\dim W \geq s$,}\\
W\cap P(q) = \emptyset & \text{ otherwise}
\end{array}
\right.
\]
for general $q\in \P^{N+1}\setminus \left( \bigcup_{i=1}^{s} X_{i}' \right)$.
Therefore
\[
\left\{
\begin{array}{ll}
\dim U\cap T \cap P(q) = \dim U\cap T -s & \text{ if $\dim T\cap U \geq s$, } \\
U\cap T \cap P(q)= \emptyset  & \text{ otherwise}
\end{array}
\right.
\]
for general $q\in \P^{N+1}\setminus \left( \bigcup_{i=1}^{s} X_{i}' \right)$.
The assertion follows by $(1)$.

$(3)$
First assume that $J(U,V) \varsubsetneq \P^{N+1}$.
Then $U\cap R(q,V) \subset V$ for $q \in \P^{N+1}\setminus J(U, V)$.

Next assume that $J(U,V)=\P^{N+1}$.
Set
\[ J_{U, V} = \overline{\bigl\{ \ (u,v,q)\in (U \times V \setminus \Delta)\times\P^{N+1} \ | \ q\in \langle u,v \rangle  \bigr\}} \subset U\times V \times \P^{N+1}.\]
Then $J_{U,V}$ is irreducible and $J_{U, V} = \dim U + \dim V +1$.
By the assumption, the projection $\pi_{3}\colon J_{U,V} \rightarrow \P^{N+1}$ is surjective.
Thus
\[
\dim \pi_3^{-1}(q) = \dim U + \dim V  -N
\]
for general $q\in \P^{N+1}$.
For every $u\in U\cap R(q, V) \setminus V$,
there exists $v\in V$ such that $v\neq u$ and $u, v, q$ are collinear.
This implies $(u,v,q) \in \pi_{3}^{-1}(q)$ and the assertion follows.

$(4)$
The incidence set
\[I_{T}= \left\{(x,q_{1},\cdots,q_{k})  \mid  q_{1}, \cdots, q_{k} \in H_{x} \right\} \subset T \times (\P^{N+1})^{k}.
\]
is irreducible and $\dim I_{T} = (N+1)k$.
For each $(q_{1},\cdots, q_{k}) \in (\P^{N+1})^{k}$, the fiber of the projection over $(q_{1}, \cdots, q_{k})$ is $T\cap  \bigcap_{j=1}^{k} P(q_{j})$ by the definition of $I_{T}$,
and $T \cap \bigcap_{j=1}^{k} P(q_{j})$ is non-empty by the assumption since $\bigcap_{j=1}^{k} P(q_{j})$ is defined by $ks$ hypersurfaces in $X'$.
Thus the projection to the last $k$ components is surjective, and generically-finite.
By the generic smoothness, 
the number of the points of $T\cap \bigcap_{j=1}^{k} P(q_{j})  = T \cap \bigcap_{i=1}^{s}\bigcap_{j=1}^{k}P^{i}(q_{j})$ is the same as the intersection number
\[T\cdot \prod_{i=1}^{s}\prod_{j=1}^{k}(P^{i}(q_{j})|_{X'})= \left(\prod_{i=1}^{s}(d_{i}-1)^{k}\right) \deg T
\]
for general $(q_{1},\cdots, q_{k}) \in \left(\P^{N+1}\setminus \left( \bigcup_{i=1}^{s} X_{i}' \right) \right)^{k}$.
\end{proof}

To close this step, we prove that $R(q,T)$ has the expected dimension and $\R(q, T)$ is well-defined as a cycle
for a closed subvariety $T\subset X$ of dimension $\geq s$,
if $T\cap \Sing(X)=\emptyset$ and $q\in \P^{N+1}\setminus \left(\bigcup_{i=1}^{s}X_{i}'\right)$ is general.
By Lemma \ref{lem} $(2)$, $\dim T\cap R(q, T) = \dim T -s$ for general $q\in \P^{N+1}\setminus \left( \bigcup_{i=1}^{s} X_{i}' \right)$.
Since $T$ is a hyperplane section of $C(q,T)$ and $R(q, T)$ is locally defined by $s$ elements in $C(q,T)$, 
$\dim R(q, T) = \dim T -s + 1$.
Hence the assertion follows.

Thus if $\beta$ is a pure-dimensional cycle  of dimension $\geq s$ on $X$ such that $|\beta| \cap \Sing(X)= \emptyset$ and $q$ is general,
$\R(q, \beta)$ is a well-defined pure-dimensional cycle on $X'$
and
\[
\dim \R(q, \beta) = \dim \beta - (s-1),\ \deg\R(q, \beta) = \left(\prod_{i=1}^{s}(d_{i}-1)\right) \deg \beta.
\]
\\

{\it Step 2: }
Now we start the proof of Proposition.
Let $S\subset X$ be a closed subvariety of dimension $ks$ with $S\cap \Sing (X)= \emptyset$.
We may assume that $S$ is contained in the support of $\alpha$.
We construct multiple residual intersections from $S$.
For each $j=0, 1,\cdots, k$, we inductively define $\R_{j}$ and its support $R_{j}$ as follows:
Set $\R_{0} = [S]$ and $R_{0}=S$.
Assume that we have constructed a pure-dimensional cycle $\R_{j-1}$ on $X$ with support $R_{j-1}=|\R_{j-1}|$
such that
\[
\dim \R_{j-1} = ks - (j-1)(s-1), \, \deg \R_{j-1} = \left(\prod_{i=1}^{s}(d_{i}-1)^{j-1}\right) \deg S
\]
and $R_{j-1} \cap \Sing(X) = \emptyset$.
Choose a point $q_{j}\in \P^{N+1}\setminus \left(\bigcup_{i=1}^{s}X_{i}'\right)$ 
so that the following conditions are all satisfied:
\begin{enumerate}
\item[$(C_{1})$] $\R(q_{j}, \R_{j-1})$ is a well-defined cycle, 
\item[$(C_{2})$] $\pi_{p}(R(q_{j}, \R_{j-1}))\cap \Sing(X)=\emptyset$ if $X$ is singular, where $\pi_{p} \colon \P^{N+1} \dashrightarrow \P^{N}$ is the linear projection from $p$
and $R(q_{j}, \R_{j-1})$ is the support of $\R(q_{j}, \R_{j-1})$,
\item[$(C_{3})$] Lemma \ref{lem} $(1)$ holds for every irreducible component $T$ of $R_{j-1}$ and $q=q_{j}$,
\item[$(C_{4})$] Lemma \ref{lem} $(2)$ holds for every irreducible component $T$ of $R_{j-1}$, every irreducible component $U$ of $|\alpha|$ and $q=q_{j}$,
\item[$(C_{5})$] Lemma \ref{lem} $(3)$ holds for every irreducible component $U$ of $R_{j-1}$, every irreducible component $V$ of $C(p, |\alpha|)$ and $q=q_{j}$,
\item[$(C_{6})$] $q_{j} \in \P^{N+1}\setminus \P^{N}$,
\item[$(C_{7})$]$(q_{1},\cdots, q_{j}) \in \pi_{1,\cdots, j}(P^{0})$, where $P^{0}$ is an open subset of $\left(\P^{N+1}\setminus \left( \bigcup_{i=1}^{s} X_{i}' \right) \right)^{k}$ such that Lemma \ref{lem} $(4)$ holds for $T=S$ and every point in $P^{0}$,
and $\pi_{1,\cdots, j} \colon (\P^{N+1})^{k}\rightarrow (\P^{N+1})^{j}$ is the projection to the first $j$ components.
\end{enumerate}
If $X$ is singular, ($C_{2}$) is satisfied for a general $q_{j}$ since $J(R_{j-1}, \Sing(X))$ is strictly contained in $\P^{N}$ by the assumption.
Set
\[\R_{j} = (\pi_{p})_{\ast} \R(q_{j}, \R_{j-1}).
\]
Then $\R_{j}$ is a pure-dimensional cycle on $X$ with support $R_{j} = |\R_{j}|$ such that
\[\dim \R_{j} = ks - j(s-1), \, \deg \R_{j} = \left(\prod_{i=1}^{s}(d_{i}-1)^{j}\right) \deg S\]
and $R_{j}\cap \Sing(X) = \emptyset$.
In particular, $\R_{k}$ is a pure-dimensional cycle on $X$ with support $R_{k}=|\R_{k}|$ such that
\[\dim \R_{k}=k, \, \deg \R_{k} =  \left(\prod_{i=1}^{s}(d_{i}-1)^{k}\right) \deg S\]
and $R_{k}\cap \Sing(X) = \emptyset$.

\begin{lem}\label{lem2}
The following hold.
\begin{enumerate}
\item[$(1)$] $\alpha$ and $\R_{k}$ intersect properly on $X$, i.e. $\dim |\alpha| \cap R_{k} = 0$.
\item[$(2)$] $S\cap R_{k}$ contains at least $\deg \R_{k}$ distinct points.
\end{enumerate}
\end{lem}

\begin{proof}

$(1)$
We claim that
\[
\dim |\alpha| \cap R_{j} = (k-j)s 
\]
for $j=0,\cdots,k$.
We prove this by induction on $j$.
The assertion is clear for $j=0$, so let $j\geq 1$.
It follows that $\dim |\alpha| \cap R_{j} = \dim C(p, |\alpha|) \cap R(q_{j}, \R_{j-1})$
since $|\alpha| \cap R_{j} = \pi_{p}(C(p, |\alpha|) \cap R(q_{j}, \R_{j-1}))$, 
thus it is enough to show that $\dim C(p,|\alpha|) \cap R(q_{j}, \R_{j-1}) = (k-j)s$.

First we note that  $R_{j-1} \cap R(q_{j}, \R_{j-1}) = \bigcup T \cap R(q_{j}, T)$ as a set, where $T$ runs all the irreducible components of $R_{j-1}$.
This holds since $T_{1} \cap R(q_{j}, T_{2}) \subset T_{2}$ for any distinct irreducible components $T_{1}, T_{2}$ of $R_{j-1}$. 

By Lemma \ref{lem} $(2)$ and the induction hypothesis,
\[
\dim |\alpha| \cap R_{j-1} \cap R(q_{j}, \R_{j-1}) = \dim |\alpha| \cap R_{j-1} - s = (k-j)s.
\]
Thus $\dim C(p, |\alpha|) \cap R_{j-1}\cap R(q_{j}, \R_{j-1}) = (k-j)s$.
On the other hand, for each irreducible component $T$ of $R_{j-1}$,
\begin{align*}
\dim C(p, |\alpha|) \cap R(q_{j}, T) \setminus T \leq (k-j)s
\end{align*}
by Lemma \ref{lem} $(3)$.
The clam follows.

Therefore
\[
\dim |\alpha| \cap R_{k} = 0,
\]
as desired.

$(2)$
If we apply Lemma \ref{lem} $(1)$ repeatedly,
\begin{align*}
S\cap R_{k} &\supseteq \bigcap_{j=0}^{k} R_{j}\\
&\supseteq \bigcap_{j=0}^{k-1}R_{j} \cap R(q_{k}, \R_{k-1})\\
&= \bigcap_{j=0}^{k-1}R_{j} \cap P(q_{k})\\
&\supseteq \cdots\\
&\supseteq S\cap \bigcap_{j=1}^{k}P(q_{j}).
\end{align*}
The proof is done by Lemma \ref{lem} $(4)$.
\end{proof}

Since $|\alpha| \cap R_{k}$ is contained in the smooth locus of $X$, the intersection product $\alpha \cdot \R_{k}$ is well-defined.
Furthermore $\alpha \cdot \R_{k}$ is a well-defined cycle by Lemma \ref{lem2} (1).
Therefore
\begin{eqnarray*}
m \deg \R_{k} & =& \alpha \cdot \R_{k} \\
& \geq& \sum_{t\in S \cap R_{k}} i(t, \alpha \cdot \R_{k}; X)\nonumber\\
& \geq &\sum_{t\in S \cap R_{k}} e_{t}(\alpha) e_{t}(\R_{k})\nonumber\\ 
& \geq &e_{S}(\alpha) \deg \R_{k} \nonumber
\end{eqnarray*}
by \cite[Corollary 12.4]{Ful} and Lemma \ref{lem2} (2).
The proof is done.
\end{proof}

\section{Proof of Theorem \ref{mainthm}}

For definitions of terminology about singularity theory, we follow \cite{Kol} and \cite[Section 2]{dF1}.
See \cite[Section 2]{dFM} for the definition and properties of Mather log discrepancy.

\begin{proof}
Take a complete intersection $X= X_{d_{1},\cdots,d_{s}}\subset \P^{\sum_{i=1}^{s} d_{i}}$ as in theorem.
It follows that $X$ is normal, factorial (\cite[Expos\'e XI, Corollaire 3.14]{Gro}) and $\rho_{X}=1$ by the numerical condition.
Moreover $-K_{X}$ is ample and $-K_{X} \sim c_{1}(\mathcal{O}_{X}(1))\cap [X]$ by adjunction.
Since terminality of $X$ can be proved in the same way as follows, we assume that $X$ is terminal (see the proof of \cite[Theorem 1.3]{dF3}).

Assume that $X$ is not birationally superrigid.
Then, by Noether-Fano inequality (\cite[Proposition 4]{dF4}),
there exists a positive integer $\mu$ and a movable linear system $\mathcal{L}\subset |-\mu K_{X}|$ 
such that $\can(X, \Bs(\mathcal{L}))< 1/\mu$, where $\can(X, \Bs(\mathcal{L}))$ is the canonical threshold of the pair $(X, \Bs(\mathcal{L}))$.
Let $c=\can(X, \Bs(\mathcal{L}))$.
For any $D \in \mathcal{L}$,
$D \sim \mu \cdot c_{1}(\mathcal{O}_{X}(1))\cap [X]$, so we have
\[
\dim \left\{x\in D \mid e_{x}(D)>\mu \right\}\leq \delta + s
\]
by Proposition \ref{prop}, thus
\[
\dim \left\{x\in D \mid e_{x}(D)\geq 1/c \right\}\leq \delta + s.
\]
Then, if $Z=D_{1}\cdot D_{2}$ is the complete intersection subscheme of $X$ defined by general members $D_{1}, D_{2}$ of $\mathcal{L}$,
any non-terminal center of the pair $(X, c Z)$ has at most dimension $\delta + s$ by \cite[Proposition 8.8]{dF1}.
Moreover, since $Z\sim \mu^{2} \cdot c_{1}(\mathcal{O}_{X}(1))^{2}\cap [X]$,
\[
\dim \left\{x\in Z \mid e_{x}(Z) > \mu^{2} \right\}\leq \delta + 2s
\]
by Proposition \ref{prop}, thus
\[
\dim \left\{x\in Z \mid e_{x}(Z) \geq 1/c^{2} \right\}\leq \delta + 2s.
\]

Take a general point $P$ in a non-terminal center of the pair $(X, cZ)$.
We cut down by $\delta + s$ general hyperplanes through $P$.
Let $\P^{\sum_{i=1}^{s}d_{i}-\delta -s} \subset \P^{\sum_{i=1}^{s}d_{i}}$ be a general linear subspace of codimension $\delta +s$ passing through $P$,
and let $W\subset \P^{\sum_{i=1}^{s}d_{i} -\delta -s }$ be the restriction of $X$ to this subspace.
By inversion of adjunction (\cite[Theorem 1.1]{EM}), $(W, cZ|_{W})$ is terminal away from finitely many points, and not terminal at $P$.
We cut down by additional $s$ general hyperplanes through $P$.
Let $\P^{\sum_{i=1}^{s}d_{i}-\delta -2s}$ $\subset \P^{\sum_{i=1}^{s}d_{i} - \delta -s}$ be a general linear subspace of codimension $s$ passing through $P$,
and let $Y\subset \P^{\sum_{i=1}^{s}d_{i} -\delta-2s}$ be the restriction of $W$ to $\P^{\sum_{i=1}^{s}d_{i}-\delta -2s}$ and $B=Z|_{Y}$.
By Proposition \ref{mult}, inversion of adjunction, adjunction formula and the assumption, it follows that
\begin{enumerate}
\item[$(1)$] $\dim \left\{x\in B \mid e_{x}(B)\geq 1/c^{2} \right\}\leq 0$.
\item[$(2)$] the pair $(Y, cB)$ is not Kawamata log terminal (klt), but klt outside $P\in B$,
\item[$(3)$] $K_{Y}\sim (\delta + 2s -1) \cdot c_{1}(\mathcal{O}_{Y}(1))\cap [Y]$,
\item[$(4)$] $m_{Y,P}^{\nu-1}\subset \overline{\Jac_{Y}}$ in $\mathcal{O}_{Y, P}$,
\end{enumerate}
where $\overline{\Jac_{Y}}$ is the integral closure of the Jacobian ideal $\Jac_{Y}$ of $Y$.

The remaining part of the proof is only a modification of the proof of \cite[Theorem]{dF2} and \cite[Theorem 5.2]{dF3}.
We briefly explain the sketch for the convenience of the reader.

Firstly, the condition $(1)$ combined with \cite[Theorem 0.1]{dFEM1} implies that the pair $(Y, 2cB)$ is klt in dimension $1$ (see \cite[Lemma 1]{dF2} and \cite[Lemma 5.3]{dF3}).
Secondly, the condition $(2)$ combined with inversion of adjunction implies that there exists a prime divisor $E$ over $Y$ with center $P$ and 
\[
a_{E}(Y, cB+ (\delta + 2s )P)\leq 0,
\]
such that the center of $E$ on the blow-up of $Y$ at $P$ has dimension $\geq \delta + 2s$ (see \cite[Lemma 2]{dF2} and \cite[Lemma 5.4]{dF3}).
If we use these facts with the inequality $c<1/\mu$ and the conditions $(3)$ and $(4)$, we have the following upper- and lower-bound of the value $\lambda = \frac{\val_{E}(P)}{c \val_{E}(B)}$:
\[
\left(1- \frac{2}{\sqrt{\prod_{i=1}^{s}d_{i}}}\right) \frac{1}{\sum_{i=1}^{s}d_{i} - 5s -2\delta - \nu - 1}>\lambda >\frac{1}{\sum_{i=1}^{s}d_{i}+1}
\]
(see \cite[Lemma 3 and 4]{dF2}, \cite[Lemma 5.5]{dF3} and the last part of the proof of \cite[Theorem 5.2]{dF3}).
Note that Nadel's vanishing and \cite[Theorem 0.1]{dFEM1} are essential in this step. 
We use \cite[Theorem 2.5]{dFM} instead of \cite[Theorem 0.1]{dFEM1} when $Y$ is singular, which gives the similar inequality between Mather log discrepancies.
Then
\[
5s + 2\delta  + \nu + 2 > \frac{2(\sum_{i=1}^{s}d_{i} +1)}{\sqrt{\prod_{i=1}^{s}}d_{i}},
\]
which contradicts to our numerical assumption.
The proof is done.
\end{proof}

\end{document}